\documentclass[11pt]{amsart}
\usepackage{geometry}                % See geometry.pdf to learn the layout options. There are lots.
\usepackage{graphicx}
\usepackage{amssymb}
\usepackage{epstopdf}
\newcommand{\R}{\mathbb{R}}
\newcommand{\N}{\mathbb{N}}
\newcommand{\am}{\mathrm{am}}
\newcommand\vd{\operatorname{visdist}}

\newtheorem{theorem}{Theorem}
\newtheorem{lemma}{Lemma}
\newtheorem{proposition}{Proposition}

\newtheorem{question}{Question}
\newtheorem*{thnn}{Theorem}

\theoremstyle{definition}
\newtheorem{definition}{Definition}
\theoremstyle{remark}
\newtheorem{remark}{Remark}

\newtheorem{example}{Example}

\usepackage[usenames]{color}
\usepackage{colortbl}

\begin{document}
\title{A quantitative version of the Morse lemma and quasi-isometries fixing the ideal boundary}
\author{Vladimir Shchur}
\address{Universit\'e Paris-Sud 11, F-91405 Orsay Cedex}
\address{\'Ecole normale sup\'erieure, 45 rue d'Ulm F-75230 Paris Cedex 05}
\address{vlshchur@gmail.com}
%\date{}                                           % Activate to display a given date or no date

\begin{abstract}
The Morse lemma is fundamental in hyperbolic group theory. Using exponential
contraction, we establish an upper bound for the Morse lemma that is optimal
up to multiplicative constants, which we demonstrate by presenting a
concrete example. We also prove an ``anti'' version of the Morse lemma. We
introduce the notion of a geodesically rich space and consider applications
of these results to the displacement of points under quasi-isometries that
fix the ideal boundary.
\end{abstract}


\maketitle
\markright{A QUANTITATIVE VERSION OF THE MORSE LEMMA AND .\,.\,.}

\section{Introduction}

Roughly speaking, the Morse lemma states that in a hyperbolic metric space,
a $\lambda$-quasigeodesic $\gamma$ belongs to a $\lambda^2$-neighborhood of
every geodesic $\sigma$ with the same endpoints. Our aim is to prove the
optimal upper bound for the Morse lemma.

\begin{theorem}[Morse lemma]\label{MainTheorem}
Let $\gamma$ be a $(\lambda,c)$-quasi-geodesic in a $\delta$-hyperbolic space
$E$ and $\sigma$ be a geodesic segment connecting its endpoints. Then $\gamma$
belongs to an $H$-neighborhood of $\sigma$, where
$$
H=\lambda^2(A_1c+A_2\delta),
$$
where $A_1$ and $A_2$ are universal constants.
\end{theorem}

We prove this theorem with $A_1=4\cdot78=312$ and
$$
A_2=4 \biggl(78+\frac{133}{\ln2}e^{157\ln2/28}\biggr)
$$
in Section~\ref{PrfMrsLem}. This result is optimal up to the value of these
constants, i.e., there exists an example of a quasi-geodesic such that $H$ is
the distance of the farthest point of $\gamma$ from $\sigma$ (see
Section~\ref{lastPr}).

The Morse lemma plays an important role in the geometry of hyperbolic
spaces. For example, it is used to prove that hyperbolicity is invariant
under quasi-isometries between geodesic spaces~\cite{ghys} (see Chapter~5.2,
Theorem~12): let $E$ and $F$ be $\delta_1$- and $\delta_2$-hyperbolic
geodesic spaces. If there exists a $(\lambda,c)$-quasi-isometry between
these two spaces, then
$$
\delta_1\le8\lambda(2H+4\delta_2+c).
$$

Hyperbolic metric spaces have recently appeared in discrete mathematics and
computer science (see, e.g.,~\cite{chepoi}). The $\delta$-hyperbolicity
turns out to be more appropriate than other previously used notions of
approximation by trees (e.g., tree width). This motivates our search for
optimal bounds for a cornerstone of hyperbolic group theory like the Morse
lemma.

Gromov's quasi-isometry classification problem for groups~\cite{gromov1}
provides another motivation. When two groups are shown to be
non-quasi-isometric, it would be desirable to give a quantitative
measure of this, such as a lower bound on the distortion of maps between
balls in these groups (we thank Itai Benjamini for bringing this issue to
our attention). We expect our optimal bound in the Morse lemma to be
instrumental in proving such lower bounds. As an indication of this, we
show that the center of a ball in a tree cannot be moved very far by a
self-quasi-isometry.

\begin{proposition}\label{ballCenter}
Let $O$ be a center of a ball of radius $R$ in a $d$-regular metric tree
$T$ ($d\ge3$). Let $f$ be $(\lambda,c)$-self-quasi-isometry of this ball.
Then for any image $f(O)$ of the center $O$,
$$
d(f(O),O)\le\min\{R,H+c+\lambda(c+1)\}.
$$
\end{proposition}

Because $\delta=0$ for a tree, we have $d(f(O),O)\le2A_1\lambda^2c$ for
sufficiently large $\lambda$. We prove this proposition in
Section~\ref{lastPr}.

We present an example of a $(\lambda,c)$-quasi-isometry of a ball in a
$d$-reqular tree that moves the center a distance $\lambda c$. We are
currently unable to fill the gap between $\lambda c$ and $\lambda^2c$.

We give a second illustration. In certain hyperbolic metric spaces,
self-quasi-isometries fixing the ideal boundary move points a bounded
distance. Directly applying the Morse lemma yields a bound of
$H\sim\lambda^2c$, while the examples that we know achieve merely
$\lambda c$. For this problem, we can fill the gap partially. Our argument
relies on the following theorem, which we call the anti-Morse lemma.

\begin{theorem}[anti-Morse lemma]\label{SecondTheorem}
Let $\gamma$ be a $(\lambda,c)$-quasi-geodesic in a $\delta$-hyperbolic
metric space and $\sigma$ be a geodesic connecting the endpoints of
$\gamma$. Let $4\delta\ll\ln\lambda$. Then $\sigma$
belongs to a $A_3(c{+}\delta)\ln\lambda$-neighborhood of $\gamma$, where
$A_3$ is some constant.
\end{theorem}

We prove Theorem~\ref{SecondTheorem} in Section~\ref{antiMorseSection}.
In Section~\ref{idBoundary}, we define the class of geodesically rich
hyperbolic spaces (it contains all Gromov hyperbolic groups), for which we
can prove the following statement.

\begin{theorem}\label{GRSTheorem}
Let $X$ be a geodesically rich $\delta$-hyperbolic metric space and $f$ be a
$(\lambda,c)$-self-quasi-isometry fixing the boundary $\partial X$. Then for
any point $O\in X$, the displacement $d(O,f(O))\le\max\{r_0, (A_4+c)\lambda\ln\lambda\}$,
where $r_0, A_4$ are constants depending on the space $X$.
\end{theorem}

We first discuss the geometry of hyperbolic spaces and prove a lemma on the
exponential contraction of lengths of curves with projections on geodesics.
We then discuss the invariance of the $\Delta$-length of geodesics under
quasi-isometries. Using these results, we prove the quantitative version of
the Morse and anti-Morse lemmas. We define the class of geodesically rich
spaces; for this class, we estimate the displacement of points by
self-quasi-isometries that fix the ideal boundary. Finally, we show that
this class includes all Gromov hyperbolic groups.

\section{The geometry of $\delta$-hyperbolic spaces}

Let $E$ be a metric space with the metric $d$. We also write $|x-y|$ for the
distance $d(x,y)$ between two points $x$ and $y$ of the space $E$. For a
subset $A$ of $E$ and a point $x$, $d(x,A)$ denotes the distance from $x$ to
$A$.

There are several equivalent definitions of hyperbolic metric spaces. We
first present the most general definition, given by
Gromov~\cite{gromov},~\cite{ghys}, although another definition is more
convenient for us.

\begin{definition}
Gromov's product of two points $x$ and $y$ at a point $z$ is
$$
(x,y)_p=\frac12(|x-p|+|y-p|-|x-y|).
$$
\end{definition}

\begin{definition}\label{defDHG}
A metric space $E$ with a metric $d$ is said to be $\delta$-hyperbolic if
for every four points $p$, $x$, $y$, and $z$,
$$
(x,z)_p\ge\min\{(x,y)_p,(y,z)_p\}-\delta.
$$
\end{definition}

\begin{definition}
A geodesic (geodesic segment, geodesic ray) $\sigma$ in a metric space $E$
is a isometric embedding of a real line (real interval $I$, real half-line
$\R_+$) in $E$.
\end{definition}

We write $xy$ for a geodesic segment between two points $x$ and $y$ (in
general, there could exist several geodesic paths between two points; we
assume any one of them by this notation). A geodesic triangle $xyz$ is a
union of three geodesic segments $xy$, $yz$, and $xz$.

\begin{definition}
A geodesic triangle $xyz$ is said to be $\delta$-thin if for any point
$p\in xy$,
$$
d(p,xz\cup yz)\le\delta.
$$
\end{definition}

A geodesic metric space is a space such that there exists a geodesic segment
$xy$ between any two points $x$ and $y$. It can be easily shown that for a
geodesic space, Definition~\ref{defDHG} is equivalent to the following
definition.

\begin{definition}\label{defDHT}
A geodesic metric space $E$ is $\delta$-hyperbolic if and only if every
geodesic triangle is $\delta/2$-thin (hereafter, we omit the factor $1/2$).
\end{definition}

According to Bonk and Schramm~\cite{bonk}, every $\delta$-hyperbolic metric
space embeds isometrically into a geodesic $\delta$-hyperbolic metric space.
Without loss of generality, we therefore consider only geodesic
$\delta$-hyperbolic spaces in what follows.

\begin{definition}
In a metric space, a \emph{perpendicular} from a point to a curve (in
particular, a geodesic) is a shortest path from this point to the curve.
\end{definition}

Of course, a perpendicular is not necessarily unique.

\begin{lemma}\label{lemAC}
In a geodesic $\delta$-hyperbolic space, let $b$ be a point and $\sigma$ be
a geodesic such that $d(b,\sigma)=R$. Let $ba$ be a perpendicular from $b$
to $\sigma$, where $a\in\sigma$. Let $c$ be a point of $\sigma$ such that
$|b-c|=R+2\Delta$. Then $|a-c|\le2\Delta+4\delta$.
\end{lemma}

\begin{figure}[!ht]
\center{\includegraphics[width=0.5\linewidth]{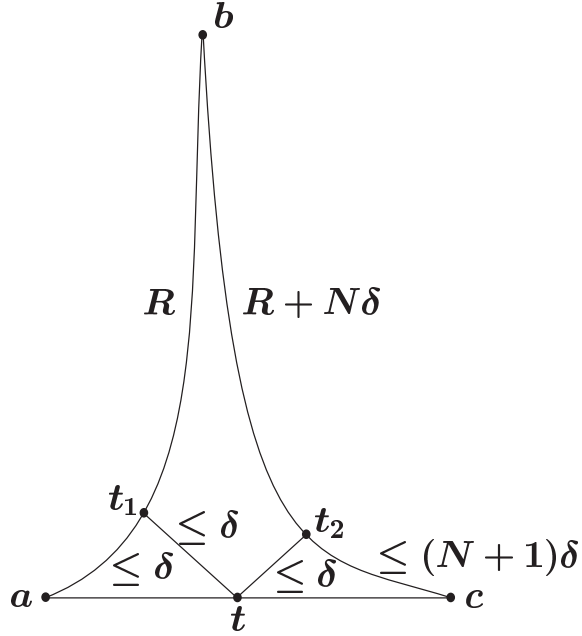}}
%\center{\includegraphics[width=0.5\linewidth]{orthPr1.eps}}
\caption{Illustration for Lemma~\ref{lemAC}.}
\label{orthPr1f}
\end{figure}

\begin{proof}
The triangle $abc$ (see Fig.~\ref{orthPr1f}) is $\delta$-thin by the
definition of a $\delta$-hyperbolic space. Hence, there exists a point
$t\in\sigma$ such that $d(t,ba)\le\delta$ and $d(a,bc)\le\delta$. Let $t_1$
and $t_2$ be the respective projections of $t$ on $ba$ and $bc$. By
hypothesis, $R$ is the minimum distance from $b$ to the points of $\sigma$.
Therefore, $R=|b-a|\le|b-t_1|+|t_1-t|\le|b-t_1|+\delta$ and
$R\le|b-t_2|+|t_2-t|\le|b-t_2|+\delta$. Hence, $|a-t_1|\le\delta$ and
$|c-t_2|\le 2\Delta+\delta$. By the triangle inequality, we obtain
$|a-c|\le|a-t_1|+|t_1-t|+|t-t_2|+|t_2-c|\le2\Delta+4\delta$.
\end{proof}

\begin{remark}
In particular, all the orthogonal projections of a point to a geodesic lie
in a segment of length $4\delta$.
\end{remark}

\begin{lemma}\label{orthPr}
In a $\delta$-hyperbolic space, let two points $b$ and $d$ be such that
$|b-d|=\Delta$. Let $\sigma$ be a geodesic and $a$ and $c$ be the respective
orthogonal projections of $b$ and $d$ on $\sigma$. Let $|a-b|>3\Delta+
6\delta$, and let $d(d,\sigma)>d(b,\sigma)$. Let two points $x_1\in ab$ and
$x_4\in cd$ be such that $2\Delta+5\delta<d(x_1,\sigma)=d(x_4,\sigma)<
|a-b|-(\Delta+2\delta)$. Then $|x_1-x_4|\le4\delta$ and $|a-c|\le8\delta$.
\end{lemma}

\begin{figure}[!ht]
\center{\includegraphics[width=0.5\linewidth]{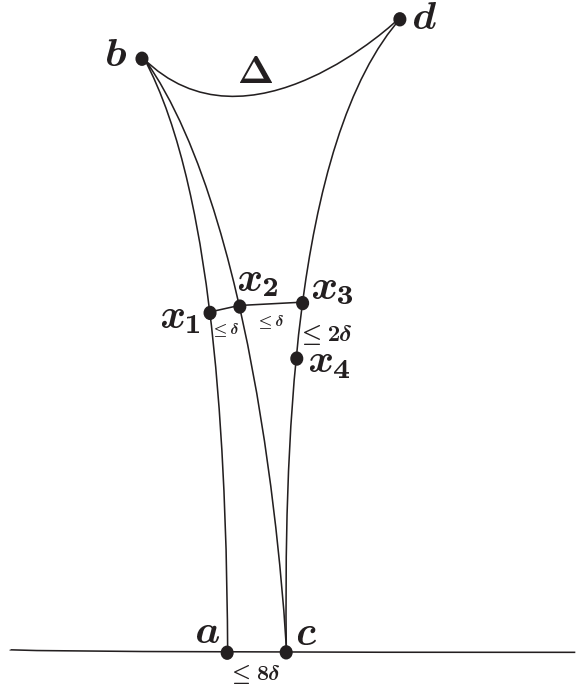}}
%\center{\includegraphics[width=0.5\linewidth]{orthPr2.eps}}
\caption{Illustration for Lemma~\ref{orthPr}.}
\label{orthPr2}
\end{figure}

\begin{proof}
(See Fig.~\ref{orthPr2}.) By the triangle inequality and because $cd$ is a
perpendicular to $\sigma$, $|c-d|\le|a-b|+|b-d|$, whence $|b-c|\le|c-d|+
|b-d|\le|a-b|+2|b-d|$. By Lemma~\ref{lemAC}, $|a-c|\le2\Delta+4\delta$. The
triangle $abc$ is $\delta$-thin, $|a-x_1|>|a-c|+\delta$. Therefore, by the
triangle inequality, $d(x_1,ac)>\delta$, and hence $d(x_1,bc)\le\delta$. Let
$x_2$ denote the point of $bc$ nearest $x_1$. Because the triangle $bcd$
is also $\delta$-thin and $|b-x_2|\ge|b-x_1|-|x_1-x_2|\ge\Delta+\delta$,
there exists a point $x_3\in cd$ such that $|x_3-x_3|\le\delta$. It
follows from the triangle $cx_1x_3$ that $|x_3-c|\ge|x_1-c|-2\delta\ge
|x_1-a|-2\delta$. On the other hand, because $x_5c$ is a perpendicular to
$\sigma$, $|x_3-c|\le |x_3-x_1|+|x_1-a|$. Now, $|a-x_1|=|c-x_4|$, and hence
$|x_4-x_3|\le2\delta$. Finally, we obtain the statement in the lemma:
$|x_1-x_4|\le4\delta$.

By the triangle inequality and because $d(x_1,\sigma)=d(x_4,\sigma)$, we
have $|x_1-c|\le|c-x_4|+|x_4-x_1|\le|a-x_1|+4\delta$. Hence, using
Lemma~\ref{lemAC}, we conclude that $|a-c|\le8\delta$.
\end{proof}

\begin{lemma}\label{orthPr1}
Let $\sigma$ be a geodesic segment, $a$ be a point not on $\sigma$, and $c$
be a projection of $a$ on $\sigma$. Let $b\in\sigma$ be arbitrary, and
let $d$ denote the projection of $b$ on $ac$. Then the $|c-d|\le2\delta$.
\end{lemma}

\begin{proof}
By hypothesis, $bd$ minimizes the distance from any its points to $ac$, and
because the triangle $bcd$ is $\delta$-thin, there exists a point $e\in bd$
such that $d(e,ac)=|e-d|\le\delta$ and $d(e,bc)\le\delta$. Because $ac$
is a perpendicular to $\sigma$, $|a-c|\le|a-d|+|d-e|+d(e,bc)\le|a-d|+
2\delta$. Hence $|c-d|\le2\delta$.
\end{proof}

\begin{lemma}\label{orthTriangle}
As in the preceding lemma, let $\sigma$ be a geodesic segment, $a$ be a
point not on $\sigma$, $c$ be a projection of $a$ on $\sigma$, and $b$ be
some point on $\sigma$. Let $d$ denote a point on $ac$ such that
$|d-c|=\delta$ and $e$ denote a point on $bc$ such that $|e-c|=3\delta$.
Then
\begin{itemize}
\item $d(d,ab)\le\delta$, $d(e,ab)\le\delta$, $d(c,ab)\le2\delta$, and
\item the length of $ab$ differs from the sum of the lengths of the two
other sides by at most $8\delta$,
$$
|a-c|+|b-c|-2\delta\le|a-b|\le|a-c|+|b-c|+8\delta.
$$
\end{itemize}
\end{lemma}

\begin{proof}
The triangle $abc$ is $\delta$-thin. Therefore, obviously, $d(d,ab)\le
\delta$ (the distance from a point of $ac$ to $ab$ is a continuous
function). We take a point $x\in bc$ such that $d(x,ca)\le\delta$.
Using Lemma~\ref{orthPr1}, we obtain $|b-x|+d(x,ca)\ge|b-c|-2\delta$, and
hence $|c-x|\le d(x,ca)+2\delta\le3\delta$.

We now let $d_1$ and $e_1$ denote the respective projections of $d$ and $e$
on $ab$. Then by the triangle inequality, we have
\begin{itemize}
\item $|a-d|-\delta\le|a-d_1|\le|a-d|+\delta$,
\item $|b-e|-\delta\le|b-e_1|\le|b-e|+\delta$, and
\item $0\le|d_1-e_1|\le|d_1-d|+|d-c|+|c-e|+|e-e_1|\le6\delta$.
\end{itemize}
Combining all these inequalities, we obtain the second point in the lemma.
\end{proof}

\begin{lemma}\label{orthPlane}
Let $\sigma$ be a geodesic and $a$ and $b$ be two points not on $\sigma$.
Further, let $a$ and $b$ have a common projection $c$ on $\sigma$. Let $d$
be a point of $\sigma$ and $c_1$ be the projection of $d$ on $ab$. Then
$$
|d-c|\le |d-c_1|+6\delta.
$$
\end{lemma}

\begin{figure}[!ht]
\center{\includegraphics[width=0.5\linewidth]{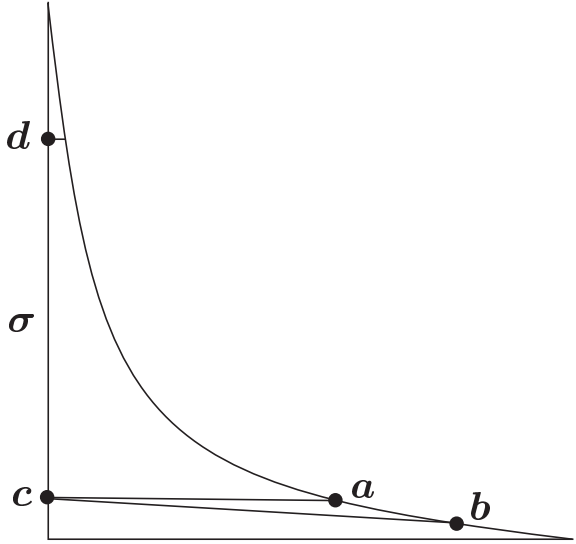}}
%\center{\includegraphics[width=0.5\linewidth]{remarkOrthPlane.eps}}
\caption{Illustration for Remark~\ref{remarkOrthPlane}.}
\label{remarkOrthPlaneFig}
\end{figure}

\begin{remark}\label{remarkOrthPlane}
Lemma~\ref{orthPlane} deals with a geodesic segment. The statement is not
true for a complete geodesic passing through $a$ and $b$, as can be seen
from Fig.~\ref{remarkOrthPlaneFig}.
\end{remark}

\begin{figure}[!ht]
\center{\includegraphics[width=0.5\linewidth]{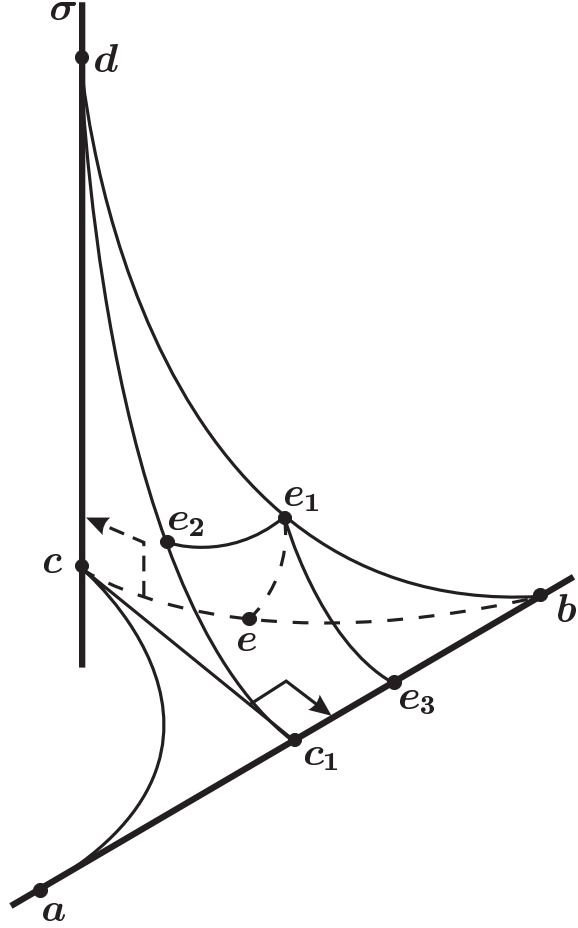}}
%\center{\includegraphics[width=0.5\linewidth]{orthPlane.eps}}
\caption{Illustration for Lemma~\ref{orthPlane}.}
\label{orthPlaneFig}
\end{figure}

\begin{proof}
We take a point $e\in bc$ such that $|c-e|=\delta$ and consider the
triangle $bcd$ (see Fig.~\ref{orthPlaneFig}). Because $bc$ is a
perpendicular to $dc$, $d(e,bd)\le\delta$. Let $e_1$ denote a projection of
$e$ on $bd$. Let $e_2$ and $e_3$ be the respective projections of $e_1$ on
the geodesic segments $dc_1$ and $bc_1$. Because the triangle $dbc_1$ is
$\delta$-thin, either $|e_1-e_2|\le\delta$ or $|e_1-e_3|\le\delta$.

{\bf I.} If $|e_1-e_2|\le\delta$, then $|d-c|\le|c-e|+|e-e_1|+
|e_1-e_2|+|e_2-d|\le|d-c_1|+3\delta$.

{\bf II.} If $|e_1-e_2|>\delta$, then the length of the path $cee_3$ is
at most $3\delta$. We apply the same arguments to $ad$ (we assume that this
is possible; otherwise, we could apply the first case to it). We obtain the
points $g$, $g_1$, and $g_3$ and the length of the path $cgg_3$ is also at
most $3\delta$. If neither of these paths intersects $cc_1$, then its length
does not exceed $6\delta$ (which follows from consideration of the triangle
$ce_3g_3$).
\end{proof}

\begin{lemma}\label{thinTriangle}
Let $E$ be a $\delta$-hyperbolic metric space and $abc$ be a triangle in
$E$. Then the diameter of the set $S$ of points of the side $ab$ such that
distance to $bc$ and $ac$ does not exceed $2d$ is not greater than
$C(d+\delta)$, where $C$ is a constant.
\end{lemma}

\begin{proof}
Let $x$ be a point of $ab$ such that $d(x,bc)\le\delta$ and $d(x,ac)\le
\delta$ and $y$ be a point of $ab$ such that $d(y,bc)\le d$ and $d(y,ac)<d$.
Without loss of generality, we assume that $y\in(a,x)$. Because the triangle
$abc$ is $\delta$-thin, one of these two distances does not exceed $\delta$.

We first assume that $d(y,ac)\le\delta$. Let $x'$ and $y'$ be points of $ac$
such that $d(x,x')\le\delta$ and $d(y,y')\le\delta$. We let $t$, $t'$, $s$,
and $s'$ denote the respective projections of $x$, $x'$, $y$, and $y'$ on
$bc$. Because $x't'$ is a perpendicular to $bc$, $|x'-t'|\le|x'-x|+|x-t|\le
2\delta$, and hence $|t-t'|\le4\delta$. If $y$ and $y'$ are sufficiently far
from $bc$, i.e., if $d\ge9\delta$, then $|s-s'|\le6\delta$ by
Lemma~\ref{orthPr}. Otherwise, we can give a rough estimate by the triangle
inequality: $|s-s'|\le|s-y|+|y-y'|+|y'-s'|\le19\delta$. Hence, in any case,
$|s-s'|\le19\delta$. We consider two cases.

If $s$ is in the segment $[b,t']$, then by applying the triangle inequality
several times, we obtain
$$
|b-y|\le|b-s|+|s-y|\le|b-t'|+|s-y|\le
|b-x|+|x-t|+|t-t'|+|s-y|\le|b-x|+5\delta+d.
$$
And because $|b-y|=|b-x|+|x-y|$, we have $|x-y|\le5\delta+d$.

The same arguments we apply if $s\in[t',c]$. We merely note that we can
replace $y$ with $y'$ and $t$ with $t'$ with respective errors less than
$\delta$ and $19\delta$:
$$
|c-y'|\le|c-s'|+|s'-y'|\le|c-s'|+|s'-y'|\le
|c-s|+19\delta+|s-y|+\delta\le|c-t'|+20\delta+d.
$$
Now, because $|c-t'|\le|c-x'|+|x'-t'|\le|c-x'|+2\delta$, we have
$$
|c-x'|+|x'-y'|=|c-y'|\le|c-x'|+22\delta+d.
$$
Finally, $|x-y|\le|y-y'|+|y'-x'|+|x-x'|\le24\delta+d$.

The case $d(y,bc)\le\delta$ is treated identically with $d$ and $\delta$
interchanged.
\end{proof}

\section{Quasi-geodesics and $\Delta$-length}

\begin{definition}\label{defQIS}
A map $f:E\to F$ between metric spaces is a {\em$(\lambda,c)$-quasi-isometry}
if
$$\frac1{\lambda}|x-y|_E-c\le|f(x)-f(y)|_F\le\lambda |x-y|_E+c
$$
for any two points $x$ and $y$ of $E$.
\end{definition}

\begin{definition}
A $(\lambda,c)$-quasi-geodesic in $F$ is a $(\lambda,c)$-quasi-isometry from
a real interval $I=[0,l]$ to $F$.
\end{definition}

Let $\gamma:I\to F$ be a curve. We assume that the interval $I=[x_0,x_n]$ of
length $|I|=l$ gives the parameterization of the quasi-geodesic $\gamma$. We
take a subdivision $T_n=(x_0,x_1,\dots,x_n)$ and let $y_i$, $i=0,1,\dots,n$,
denote $\gamma(x_i)$. The {\em mesh} of $T_n$ is
$d(T_n)=\min_{0<i\le n}|y_i-y_{i-1}|$.

\begin{definition}[$\Delta$-length]
Let $\gamma:I\to F$ be a curve. The value
$$
L_\Delta(\gamma)=\sup_{T_n:d(T_n)\ge\Delta}\sum_{i=1}^n|y_i-y_{i-1}|
$$
is called the {\em $\Delta$-length} of the quasi-geodesic $\gamma$.
\end{definition}

We note that the values of the $\Delta$-length and the classical length are
the same for a geodesic.

\begin{lemma}\label{delta1}
Let $\gamma:I\to F$ be a $(\lambda,c)$-quasi-geodesic. For $\Delta\ge2c$,
$$
L_\Delta(\gamma)\le2\lambda l.
$$
\end{lemma}

\begin{proof}
By the definition of the $\Delta$-length, $\Delta\le|y_i-y_{i-1}|\le
\lambda|x_i-x_{i-1}|+c$. Hence, because $\Delta\ge2c$, we obtain
$|x_i-x_{i-1}|\ge(\Delta-c)/\lambda\ge c/\lambda$.

Now, by the definition of a quasi-geodesic (and a quasi-isometry in
particular), we have
$$
\sup_{T_n}\sum_i|y_i-y_{i-1}|\le\sup_{T_n}\sum_i(\lambda|x_i-x_{i-1}|+c)\le
\sup_{T_n}\sum_i2\lambda |x_i-x_{i-1}|=2\lambda l,
$$
where the last equality follows because the sum of $|x_i-x_{i-1}|$ for every
subdivision of the interval $I$ is exactly equal to the length of $I$.
\end{proof}

\begin{lemma}\label{qgLen}
Let $\gamma:I\to F$ be a $(\lambda,c)$-quasi-geodesic. Let $R\ge c$ be the
distance between the endpoints of $\gamma$, and let $\Delta\ge2c$. Then
$L_\Delta(\gamma)\le4\lambda^2R$.
\end{lemma}

\begin{proof}
By the definition of a quasi-isometry, $l/\lambda-c\le R\le\lambda l+c$.
Hence, $l\le\lambda (R+c)$. And by Lemma~\ref{delta1}, $L_\Delta(\gamma)\le
2\lambda^2(R+c)$. In particular, $L_\Delta(\gamma)\le4\lambda^2R$ for
$R\ge c$.
\end{proof}

The next lemma allows replacing arbitrary quasi-geodesics with continuous
ones.

\begin{lemma}\label{contQG}
Let $\gamma$ be a $(\lambda,c)$-quasi-geodesic, and let $\Delta\ge c$. Let
$T={t_0,t_1,\dots,t_n}\subset\gamma$ be the set of points on $\gamma$ such
that $T$ gives the $\Delta$-length value $L_\Delta$.
\begin{enumerate}
\item[1.]Then the curve $\tilde{\gamma}$ consisting of the geodesic segments
$[t_i,t_{i+1}]$, $i=0,1,\dots,n-1$, is a $(\lambda,12\Delta+3c)$-geodesic
with the (classical) length $L_\Delta$.
\item[2.]Let $y$ and $y'$ be points of $\tilde{\gamma}$ such that
$d(y,y')\ge6\Delta+c$. Let $\tilde{\gamma}_0$ be the part of
$\tilde{\gamma}$ between $y$ and $y'$. Then the (classical) length of
$\tilde{\gamma}_0$ is not greater than $L_\Delta(\tilde{\gamma}_0)\le
4\lambda^2(R+6\Delta)$.
\end{enumerate}
\end{lemma}

\begin{figure}[!ht]
\center{\includegraphics[width=0.5\linewidth]{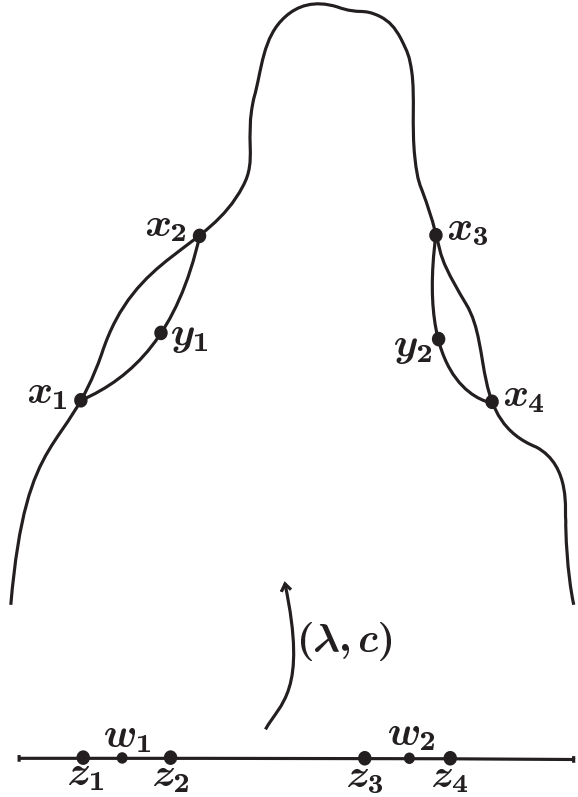}}
%\center{\includegraphics[width=0.5\linewidth]{lemma-ex-qg.eps}}
\caption{Construction of the continuous arc $\tilde{\gamma}$ from the
quasi-geodesic $\gamma$.}
\label{imageLexqg}
\end{figure}

\begin{proof}
We first note that for every $i=0,1,\dots,n-1$, the length of the interval
$|[t_i,t_{i+1}]|\le3\Delta$. Indeed, if $|[t_i,t_{i+1}]|>3\Delta$, then we
can add a point $t'_i$ to the partition $T$. Such a point exists because
the gaps on a quasi-geodesic cannot be greater than $c$.

We assume that $\gamma$ is parameterized by an interval $I$; $t_i^{-1}\in I$
are the parameters of $t_i$, $i=0,1,\dots,n$ (see Fig.~\ref{imageLexqg}).
Let $[t_i^{-1},t_{i+1}^{-1}]$ be the affine parameterization of the geodesic
segments $[t_i,t_{i+1}]$. Then the conditions for being a
$(\lambda,4c)$-geodesic are satisfied automatically for the points of the
same segment.

To simplify the notation, we let $[x_1,x_2]$ and $[x_3,x_4]$ denote two
different intervals of $\tilde{\gamma}$ and $[z_1,z_2]$ and $[z_3,z_4]$
denote their parameters. We take two points $y_1\in[x_1,x_2]$ and
$y_2\in[x_3,x_4]$, where $w_1$ and $w_2$ are their parameters. By the
triangle inequality and by the definition of a quasi-isometry,
$$
|y_1-y_2|\le|x_2-x_3|+|y_1-x_2|+|y_2-x_3|\le
|x_2-x_3|+6\Delta\le\lambda|z_2-z_3|+c+6\Delta.
$$
Similarly, we obtain the lower bound
$$
|y_1-y_2|\ge|x_2-x_3|-|y_1-x_2|-|y_2-x_3|\ge
|x_2-x_3|-6\Delta\ge\frac1{\lambda}|z_2-z_3|-c-6\Delta.
$$

By the definition of a quasi-isometry, $|z_k-z_{k+1}|\le
\lambda(|x_k-x_{k+1}|+c)\le\lambda(3\Delta+c)$ with $k=1,3$. Hence,
$$
|w_1-w_2|-2\lambda(3\Delta+c)\le|z_2-z_3|\le|w_1-w_2|.
$$
Therefore,
$$
\frac1{\lambda}|w_1-w_2|-\frac{2\lambda(3\Delta+c)}{\lambda}-
6\Delta-c\le|y_1-y_2|\le\lambda |w_1-w_2|+6\Delta+c.
$$
Consequently, $\tilde{\gamma}$ is a quasi-geodesic with the constants
$\lambda$ and $12\Delta+3c$ and statement~1 in the lemma is proved.

To prove statement~2, we need merely note that if $|y_1-y_2|\ge6\Delta+c$,
then $c\le|x_1-x_4|\le|y_1-y_2|+6\Delta$ by the triangle inequality. The
left-hand inequality allows applying Lemma~\ref{qgLen} to the part
$\gamma_0$ between $x_1$ and $x_4$ of the initial quasi-geodesic $\gamma$,
and we use the right-hand part to obtain the upper bound,
$$
L(\tilde{\gamma}_0)\le L_\Delta(\gamma_0)\le4\lambda^2(R+6\Delta).
$$
\end{proof}

\section{Exponential contraction}

\begin{lemma}[Exponential contraction]\label{expC}
Let $\Delta>0$. In a geodesic $\delta$-hyperbolic space $E$, let $\gamma$ be
a connected curve at a distance not less than $R\ge\Delta+58\delta$ from a
geodesic $\sigma$. Let $L_\Delta$ be the $\Delta$-length of $\gamma$. Let
$r=\lfloor(R-\Delta-58\delta)/19\delta\rfloor19\delta$. Then the length of
the projection of $\gamma$ on $\sigma$ is not greater than
$$
\max\biggl(\frac{4\delta}
{\Delta}e^{-Kr/\delta}(L_\Delta+\Delta),8\delta\biggr).
$$
In other words,
\begin{itemize}
\item if $R\le\Delta+58\delta+(\delta/K)\ln\bigl((L_\Delta+\Delta)/
2\Delta\bigr)$, then the length of the projection of $\gamma$ on $\sigma$ is
not greater than $(4\delta/\Delta)e^{-Kr/\delta}(L_\Delta+\Delta)$;
\item otherwise, it is not greater than $8\delta$.
\end{itemize}
\end{lemma}

\begin{figure}[!ht]
\center{\includegraphics[width=0.5\linewidth]{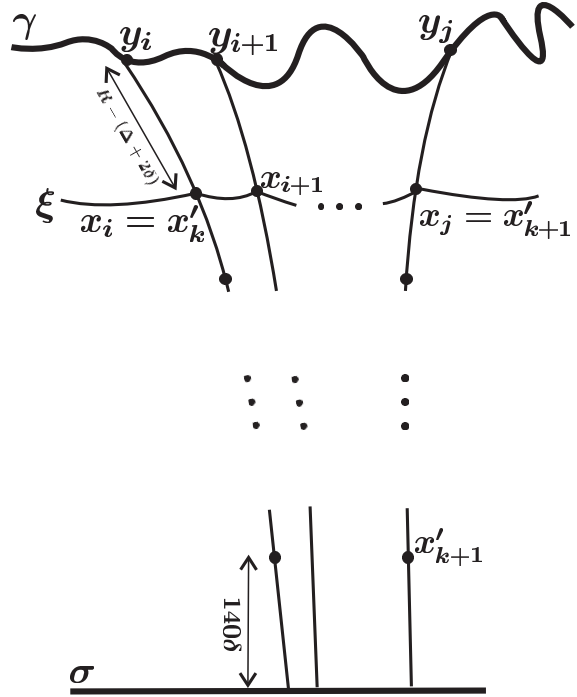}}
%\center{\includegraphics[width=0.5\linewidth]{expContr.eps}}
\caption{Exponential contraction of the length of a curve $\gamma$ under
projection on a geodesic $\sigma$.}
\label{expContr}
\end{figure}

% TODO: to remove completely?
% Here we need to discuss the meaning of the lemma for $\delta\ll1$,
% in particular when $\delta=0$.
% To begin with, we take a tree, which is a $0$-hyperbolic space.
% Obviously, every projection of a connected curve $\gamma$ on a geodesic
% $\sigma$ that does not intersect $\gamma$ is a single point, and the length
% of the projection is therefore zero. We now consider the statement in the
% lemma with $\delta\to0$: $r\to R-\Delta$, $K\to\infty$,
% $(8\delta/7\Delta)e^{-Kr}L_\Delta\to0$, and $8\delta\to0$.
% This fits nicely with the consequence of Gromov's theorem that
% hyperbolic spaces with small $\delta$ are close to trees.

\begin{proof}
Let $y_0,y_1,\dots,y_n$ be points on $\gamma$ such that $|y_i-y_{i-1}|=
\Delta$ for $i=1,2,\dots,n-1$, $|y_n-y_{n-1}|\le\Delta$, and $y_0$ and $y_n$
are the endpoints of $\gamma$. Let $y_k$ be the point of this set that is
nearest $\sigma$. We take a perpendicular from $y_k$ to $\sigma$ and a point
$x_k$ on it with $|y_k-x_k|=\Delta+3\delta$. Now, on the perpendiculars from
all other points $y_i$, we take points $x_i$ such that $d(x_i,\sigma)=
d(x_k,\sigma)$ (see Fig.~\ref{expContr}). By Lemma~\ref{orthPr},
$|x_i-x_{i-1}|\le4\delta$ for $i=1,2,\dots,n$. Therefore,
$$
\sum_{i=1}^n|x_i-x_{i-1}|\le n4\delta\le n\Delta\frac{4\delta}{\Delta}\le
\frac{4\delta}{\Delta}(L_{\Delta}+\Delta).
$$

We set $\bar{x}_0=x_0$ and $\bar{x}_{n^1}=x_n$ and select points
$\bar{x}_i\in\{x_1,x_2,\dots,x_{n-1}\}$ such that $8\delta\le|x_i-x_{i-1}|\le
16\delta$. For each $i=0,1,\dots,n^1$, we choose a perpendicular from
$\bar{x}_i$ to $\sigma$, move $\bar{x}_i$ along it a distance
$16\delta+3\delta=19\delta$ toward $\sigma$, and obtain $x^1_i$. By
Lemma~\ref{orthPr}, $|x^1_i-x^1_{i-1}|\le4\delta$ and
$$
\sum_{i=1}^{n^1}|x^1_i-x^1_{i-1}|\le n^14\delta\le
\frac12\sum_{i=1}^{n^1}|\bar{x}_i-\bar{x}_{i-1}|\le
\frac12\sum_{i=1}^n|x_i-x_{i-1}|\le
\frac12\frac{4\delta}{\Delta}(L_{\Delta}+\Delta).
$$

We can continue such a process while the distance from the set of points
$\{x^m_i,\;i=0,1,\dots,n^m\}$ to $\sigma$ is not less than $19\delta$ and
$|x^m_0-x^m_{n^m}|\ge8\delta$. After $k$ steps, we have
$$
\sum_{i=1}^{n^k}|x^k_i-x^k_{i-1}|\le
\frac1{2^k}\frac{4\delta}{\Delta}(L_{\Delta}+\Delta)=
\frac{4\delta}{\Delta}e^{-((\ln2)/19\delta)(19\delta k)}
(L_\Delta+\Delta).
$$
We set $r=19\delta k$ and $K=(\ln2)/19$. We need $8\delta\le
(4\delta/\Delta)e^{-Kr/\delta}(L_\Delta+\Delta)$ and hence $r\le(\delta/K)
\ln\bigl((L_\Delta+\Delta)/2\Delta\bigr)$. Now, if the distance between the
projections of the endpoints $|x^m_0-x^m_{n^m}|$ is not less than $8\delta$
at some step $m$, then we use Lemma~\ref{orthPr} to do the last projection
on $\sigma$, and its length does not exceed $8\delta$. Otherwise, we must do
the last descent to the distance $55\delta$ using Lemma~\ref{orthPr} (the
estimate for the projection on a geodesic with $\Delta=16\delta$ gives the
necessary distance from the set of points to the geodesic to be greater than
$3*16\delta+6\delta=54\delta$) and intervals of a length not less than
$8\delta$ contract to intervals of a length not more than $\delta$, and we
hence have a contraction factor of unity at the last step.
\end{proof}

\section{Quantitative version of the Morse lemma}

We are now ready to prove our main result. In a $\delta$-hyperbolic space
$E$, any $(\lambda,c)$-quasi-geodesic $\gamma$ belongs to an $H$-neighborhood
of a geodesic $\sigma$ connecting its endpoints, where the constant $H$ depends
only on the space $E$ (in particular, on the constant $\delta$) and the
quasi-isometry constants $\lambda$ and $c$.

\subsection{Attempts}
To motivate our method, we describe a sequence of arguments yielding sharper
and sharper estimates. We start with the proof in~\cite{ghys}, Chapter~5.1,
Theorem~6 and Lemma~ 8, where the upper bound $H\le\lambda^8c^2\delta$ was
obtained (up to universal constants, factors of the order $\log_2(\lambda c
\delta)$). The first weak step in this proof is replacing a
$(\lambda,c)$-quasi-geodesic with a discrete $(\lambda',c)$-quasi-geodesic
$\gamma'$ parameterized by an interval $[1,2,\dots,l]$ of integers, where
$\lambda'\sim\lambda^2c$. For a suitable $R\sim\lambda'^2$, we take an arc
$x_ux_v$ of $\gamma'$ and introduce a partition of that arc $x_u,x_{u+N},
x_{u+2N},\dots,x_v$ for some well-chosen $N\sim\lambda'$. The approximation
of a $\delta$-hyperbolic space by a tree (see~\cite{ghys}, Chapter~2.2,
Theorem~12.ii) is used to obtain an estimate of the form
$|y_{u+iN}-y_{u+(i+1)N}'|\le c'\sim\ln\lambda'$. By the triangle
inequality, $|x_u-x_v|\le|x_u-y_u|+|y_u-y_{u+N}|+\dots+|y_v-x_u|\le
2(R+\lambda')+(N^{-1}|u-v|+1)c'$. On the other hand,
$\lambda'^{-1}|u-v|\le|x_u-x_v|$. Combining these two inequalities, we obtain an
estimate for $|u-v|$ and hence for a distance from any point of the arc
$x_ux_v$ to the point $x_u$. The second weak step in this argument is in the
estimate of the length of projections, which can be improved significantly.

Another proof was given in~\cite{Alonso}. It allows obtaining the estimate
$\lambda^2H_{\am}$, where $H_{\am}$ is the constant of the anti-Morse lemma
(see Section~\ref{antiMorseSection}) and is given by the equation
$H_{\am}\simeq\ln\lambda+\ln H_{\am}$.\footnote{Be careful while
reading~\cite{Alonso} because a slightly different definition of
quasi-geodesics is used there with $\lambda_1=\lambda^2$;
cf.~Lemma~\ref{qgLen}.} It is very close to an optimal upper bound but still
not sharp. Also we need to notice that the sharp estimate for $H_{\am}\simeq\ln\lambda$. The proof uses the
notion of ``exponential geodesic divergence.''

\begin{definition}\label{divGeod}
Let $F$ be a metric space. We call $e:\N\to\R$ a \emph{divergence function}
for the space $F$ if for any point $x\in F$ and any two geodesic segments
$\gamma=(x,y)$ and $\gamma'=(x,z)$, the length of a path $\sigma$ from
$\gamma(R+r)$ to $\gamma'(R+r)$ in the closure of the complement of a ball
$B_{R+r}(x)$ (i.e., in $\overline{X\setminus B_{R+r}(x)}$) is not greater
than $e(r)$ for any $R,r\in\N$ such that $R+r$ does not exceed the lengths of
$\gamma$ and $\gamma'$ if $d(\gamma(R),\gamma'(R))>e(0)$.
\end{definition}

The divergence function is exponential in a hyperbolic space. The next step
is to prove the anti-Morse lemma. The authors of~\cite{Alonso} take a point
$p$ of the geodesic $\sigma$ that is the distant from the quasi-geodesic
$\gamma$ and construct a path $\alpha$ between two points of $\gamma$ such
that $\alpha$ is in the complement of the ball of radius $d(p,\gamma)$ with
the center $p$. Finally, they compare two estimates of the length: one
estimate follows from the hypothesis that $\alpha$ is a quasi-geodesic, and
the other is given by the exponential geodesic divergence. To prove the
Morse lemma, they take a (connected) part $\gamma_1$ of $\gamma$ that
belongs to the complement of the $H_{\am}$-neighborhood of the geodesic
$\sigma$, and they show that the length of $\gamma_1$ does not exceed
$2\lambda^2H_{\am}$ by the definition of a quasi-geodesic. In~\cite{Alonso},
they also use another definition of a quasi-geodesic, which is less general
than our definition because, in particular, it assumes that a quasi-geodesic
is a continuous curve. Consequently, some technical work is needed to
generalize their results.

To improve these bounds, we use Lemma~\ref{expC} (exponential contraction)
instead of exponential geodesic convergence and Lemma~\ref{qgLen}, which do
not require discretization as in~\cite{ghys} and provide a much more precise
estimate for a length of a projection. We can then take $R=\ln\lambda$ and
obtain $H\le O(\lambda^2\ln\lambda)$ by a similar triangle inequality.

Below, we prove the Morse and anti-Morse lemmas independently. We only
mention that arguments in~\cite{Alonso} can be used to deduce the optimal
bound for the Morse lemma from the anti-Morse lemma. We can also obtain an
optimal upper bound for $H$ from Lemma~\ref{lem11}.

We now sketch the proof of a stronger result (but still not optimal):
$H\le O(\lambda^2\ln^*\lambda)$, where $\ln^*\lambda$ is the minimal
number $n$ of logarithms such that
$\underbrace{\ln\dots\ln}_n\lambda\le1$.

The preceding argument is used as the initial step. It allows assuming that
the endpoints $x$ and $x'$ of $\gamma$ satisfy $|x-x'|\le O(\ln\lambda)$.
%
% Let $y_u$ and $y_v$ be two points of $\gamma$ such that the distance
% of the part of % $\gamma$ between them from $\sigma$ is not less than
% $\delta\ln\lambda$. Let $L_{uv}$ denote the $\Delta$-length of this part.
% From Lemma~\ref{expC}, the definition of a quasi-isometry, the triangle
% inequality $|y_u-y_u'|+|y_u'-y_v'|+|y_v'-y_v|\ge|y_u-y_v|$, and
% Lemma~\ref{qgLen}, we conclude that $|y_u-y_v|\le
% C_1(c,\delta)\lambda^2\ln\lambda$.
% Here, we project $\gamma$ on $\sigma$ as in the preceding proof.
%
Then comes an iterative step. We prove that if $xx'$ is an arc on $\gamma$
and $|x-x'|=d_1$, then there exist two points $y$ and $y'$ at distance at
most $C_2(c,\delta)\lambda^2$ from a geodesic $\sigma_1$ connecting $x$ and
$x'$ such that $d_2:=|y-y'|\le C_3(c,\delta)\ln d_1$. Indeed, we choose a
point $z$ of the arc $xx'$ that is farthest from $\sigma_1$ and let
$\sigma'$ denote a perpendicular from $z$ to $\sigma_1$. If all points of
the arc $xx'$ (on either side of $z$) whose projection on $\sigma'$ is at a
distance $\le\lambda^2$ from $\sigma_1$ are at a distance not less than
$\ln d_1$ from $\sigma'$, then Lemma~\ref{expC} implies that the length of
the arc is much greater than $\lambda^2\ln d_1$, contradicting the
quasi-geodesic assumption. Hence, there are points $y$ and $y'$ that are
near $\sigma'$. We can arrange that their projections on $\sigma'$ are near
each other, which yields $|y-y'|\le\ln d_1$. We apply this relation several
times starting with $d_1=C_1(c,\delta)\ln\lambda$ until $d_i\le1$ for some
$i=\ln^*\lambda$.

In summary, we use two key ideas to improve the upper bound of $H$:
exponential contraction and a consideration of a projection of $\gamma$ on a
different geodesic $\sigma'$.

\subsection{Proof of the Morse lemma}
\label{PrfMrsLem}
We use the same ideas to prove the quantitative version of the Morse lemma,
but we should do it more accurately. Let $\gamma$ be a
$(\lambda,c)$-quasi-geodesic  in a $\delta$-hyperbolic space $E$, and let
$\sigma$ be a geodesic segment connecting its endpoints. We prove that
$\gamma$ belongs to an $H$-neighborhood of $\sigma$, where
\begin{equation}
H=4\lambda^2\biggl(78c+\biggl(78+
\frac{133}{\ln2}e^{157\ln2/28}\biggr)\delta\biggr).
\end{equation}

\begin{remark}
It is easy to give an example where $H=\frac{\lambda^2c}2$ (see
Section~\ref{lastPr2}).
\end{remark}

Indeed, a path that goes back and forth along a geodesic segment of length
$\lambda^2c$ in a tree is a $(\lambda,c)$-quasi-geodesic (see
Section~\ref{lastPr} for details).

\begin{figure}[!ht]
\center{\includegraphics[width=0.5\linewidth]{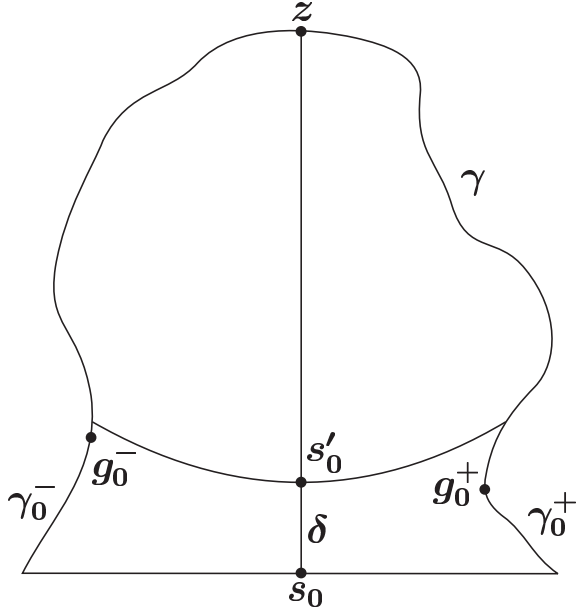}}
%\center{\includegraphics[width=0.5\linewidth]{pfThm.eps}}
\caption{Illustration of proof of Theorem~\ref{MainTheorem}}
\label{pfThm}
\end{figure}

\begin{proof}[Proof of Theorem~\ref{MainTheorem}]
Applying Lemma~\ref{contQG} to the quasi-geodesic $\gamma$ with $\Delta=2c$,
we obtain a continuous $(\lambda,27c)$-quasi-geodesic $\tilde{\gamma}$. By
Lemma~\ref{qgLen}, $\gamma$ belongs to a
$4\lambda^2\cdot6c{=}24\lambda^2c$-neighborhood of $\tilde{\gamma}$.
{\bf Hereafter, we consider only the $(\lambda,27c)$-quasi-geodesic
$\tilde{\gamma}$, which for brevity is denoted simply by $\gamma$, and we
set $\tilde{c}=27c$.} The classical length of the part of this quasi-geodesic
between two points separated by a distance $R$ does not exceed
$4\lambda^2(R+\tilde{c})$.

We introduce the following construction for subdividing the quasi-geodesic
$\gamma$. We let $z$ denote the point of our quasi-geodesic that is farthest
from $\sigma$. Let $\sigma_0=\sigma$ be the geodesic connecting the endpoints of
$\gamma$. Let $\sigma'_0$ be the geodesic minimizing the distance between
$z$ and $\sigma_0$ (because $\sigma_0$ is a geodesic segment, $\sigma'_0$
is not necessarily perpendicular to the complete geodesic carrying
$\sigma_0$). Let $s_0$ denote the point of intersection of $\sigma_0$ and
$\sigma'_0$. Let $s'_0$ be the point of $\sigma'_0$ such that the length of
the segment $[s_0,s'_0]$ is equal to $\delta$. We consider the set of points
of $\gamma$ whose projections on $\sigma'_0$ belong to the segment
$[s_0,s'_0]$. The point $z$ separates this set into two subsets $\gamma^+_0$
and $\gamma^-_0$ (see Fig.~\ref{pfThm}).

Let $d^\pm_0$ denote the minimal distance of points of $\gamma^\pm_0$ to
$\sigma'_0$. We also introduce the following notation:
\begin{itemize}
\item $d_0=d^+_0+d^-_0+\delta$;
\item $\gamma_1$ is a connected component of
$\gamma\setminus(\gamma^+_0\cup\gamma^-_0)$ containing $z$ and is also a
quasi-geodesic with the same constants and properties as $\gamma$;
\item $\sigma_1$ is a geodesic connecting the endpoints of the sub-quasi-geodesic
$\gamma_1$;
\item $L_1$ is the length of $\gamma_1$.
\end{itemize}

%$$\sum_{j=i}^ne^{-\frac{K}{2}(d_j-76\delta)}\le \frac{\lambda^2(d_i+2c)}{2}$$

Applying the same idea to the curve $\gamma_1$, the same point $z$, and the
geodesic $\sigma_1$, we obtain the geodesic $\sigma'_1$, the parts
$\gamma^\pm_1$ of the quasi-geodesic, and the distances $d^\pm_1$. We have
$l(\sigma'_0)\le l(\sigma'_1)+\delta+6\delta$. To show this, we apply
Lemma~\ref{orthPlane} assuming that $c=s'_0, d=z$, and $a$ and $b$ are the
endpoints of $\gamma_1$. Continuing the process, we obtain a subdivision of
$\gamma$ by $\gamma^\pm_i$ and two families of geodesics $\sigma_i$ and
$\sigma'_i$. Finally, for some $n$, we obtain $d_n\le\tilde{c}+
\delta+77\delta=78\delta+\tilde{c}$.

The quantity $L_i$ is the length of the subcurve $\gamma_{i-1}$, which is
also a quasi-geodesic. Hence, $l(\sigma'_n)\le L_n\le4(d_n+\tilde{c})
\lambda^2$ by construction.Therefore,
$$
l(\sigma'_0)\le\sum_{i=1}^n7\delta+4(78\delta+2\tilde{c})\lambda^2.
$$
Our goal is to prove that for sufficiently large $\lambda$,
$\sum d_i\le C\lambda^2$, where $C$ is a constant depending only on
$\tilde{c}$ and $\delta$.

Because the value of the classical length of a segment is not less then the
value of its $\Delta'$-length, by Lemma~\ref{expC} (with $\Delta'=\delta$)
and because $\lfloor(d^\pm_{i+1}-\delta-58\delta)/19\delta\rfloor
19\delta\ge d^\pm_{i+1}-78\delta$, we obtain
$$
l(\gamma^+_i\cup\gamma^-_i)\ge
\delta\frac{\delta}{4\delta}\max(e^{K(d^+_{i+1}-78\delta)/\delta},
e^{K(d^-_{i+1}-78\delta)/\delta})\ge
\frac{\delta}4e^{K(d_{i+1}-\delta-156\delta)/2\delta}.
$$
On the other hand, $l(\gamma^+_i\cup\gamma^-_i)=L_i-L_{i+1}$. Hence, setting
$C_0=(\delta/4)e^{-157K/2}$, we have
\begin{equation}\label{eq1}
C_0e^{Kd_{i+1}/2\delta}\le L_i-L_{i+1}.
\end{equation}

Let $g^{\pm}_i$ be a point of $\gamma^{\pm}_i$ that minimizes the distance
to $\sigma'_i$. The part of the quasi-geodesic $\gamma$ between $g^+_i$ and
$g^-_i$ is also a quasi-geodesic with the same constants and properties. By
the triangle inequality, $|g^-_i-g^+_i|<d^+_i+d^-_i+\delta$. Therefore,
by construction (see the beginning of the proof) and because
$d_i\ge78\delta$,
\begin{equation}\label{eq2}
L_i\le4\lambda^2(d_i+\tilde{c})\le8\lambda^2d_i.
\end{equation}

The function $e^{-d}$ is decreasing. Therefore, because
$d_i\ge\frac4{\lambda^2}L_i$, we obtain
$$
\frac K{2\delta}d_ie^{-Kd_i/2\delta}\le
\frac K{2\delta}\frac4{\lambda^2}L_ie^{-(4K/2\delta\lambda^2)L_i}.
$$

We are now ready to estimate $n$:
$$
n=\sum_{i=1}^n1=
\frac1{C_0}\sum_{i=1}^ne^{-Kd_i/2\delta}C_0e^{Kd_i/2\delta}\le
\frac1{C_0}\frac{\lambda^2\delta}{4K}
\sum_{i=1}^ne^{-(8K/2\delta\lambda^2)L_i}
\frac{4K}{\lambda^2\delta}(L_{i-1}-L_i).
$$
Setting $X_i=(4K/\lambda^2\delta)L_i$, we have
$$
\sum_{i=1}^ni\le
\frac{\lambda^2\delta}{4C_0K}\sum_{i=1}^ne^{-X_i}(X_{i-1}-X_i),
$$
and because the function $e^{-X}$ is decreasing for $X\ge0$, we can use the
estimate
$$
\sum_{i=1}^ne^{-X_i}(X_{i-1}-X_i)\le
\int_0^\infty e^{-X}dX=-e^{-x}|_0^\infty=1.
$$

Summarizing all the facts, returning to the initial notation, and
recalling that $K=\ln2/19$, we finally obtain the claimed result
$$
H=4\lambda^2\biggl(78c+\biggl(78+\frac{133}{\ln2}
e^{157\ln2/38}\biggr)\delta\biggr).
$$
\end{proof}

\section{Examples}\label{lastPr}

\subsection{Proof of Proposition~\ref{ballCenter}}

Here, we prove Proposition~\ref{ballCenter} (see the introduction). We call
any connected component of a ball with a deleted center $O$ a \emph{branch}.
We call points that are sent to the branch containing the image of the
center $f(O)$ green points and all other points of $T$ red points.

\begin{proof}[Proof of Proposition~\ref{ballCenter}]
We show that there exist two red points $r_1$ and $r_2$ such that
$d(O,r_1r_2)\le r=c+1$.

By Definition~\ref{defQIS}, a $c$-neighborhood of every point of the border
should contain a point of the image. We must have at least
$(d{-}1)d^{R-c-1}$ red points near the border (we exclude the green part).
The number of points in each connected component of the complement of the
ball of radius $r$ is less than $d^{R-r}$. Therefore, if $r\gg c$, then one
component contains an insufficient number of points to cover the border of
$B$. Hence, there exists two points $r_1$ and $r_2$ in different components
of $T$, which means that the geodesic $r_1r_2$ passes at a distance less
than $r$ from the center $O$ and the quasi-geodesic $f(r_1r_2)$ passes at
a distance $\lambda r+c$ from $f(O)$ and belongs to an $H$-neighborhood of
the geodesic $f(r_1)f(r_2)$. Because every path from $f(O)$ to
$f(r_1)f(r_2)$ passes through $O$, we conclude that $d(O,f(0))<H+c+
\lambda r$. We need only choose a good value for $r$. Simply calculating the
number of points in a mentioned component gives the estimate
$1+d+d^2+\dots+d^{R-r}\le(1/\ln d)d^{R-r+1}$. For $r=c+1$, we have
$(1/\ln d)d^{R-r+1}\le(d-1)d^{R-c-1}$, which completes the proof.
\end{proof}

\subsection{Optimality of Theorem~\ref{MainTheorem}}\label{lastPr2}
We present an example of a $(\lambda,c)$-quasi-geodesic $\gamma$ in a tree
with $H=\lambda^2c/2$. We take a real interval $[a,b]$ of length
$\lambda^2c/2$ that is a subtree. We use an interval $I=[u,v]$ of length
$\lambda c$ to parameterize $\gamma$. We define $\gamma$ as follows:
\begin{itemize}
\item $\gamma(u)=\gamma(v)=a$,
\item we set $\gamma(w)=b$ for the midpoint $w$ of $I$, and
\item we set $D=min\{|u-x|,|v-x|\}$ and $|a-\gamma(x)|=\lambda D$ for any
$x\in[a,b]$.
\end{itemize}
It is easy to verify that $\gamma$ is a well-defined quasi-geodesic. On the
half-intervals $[u,w]$ and $[w,v]$, $\gamma$ just stretches the distances by
$\lambda$. We now take any two points $x\in[u,w]$ and $y\in[w,v]$. Assuming
that $|u-x|\le|v-y|$, we obviously have $|x-y|=|u-v|-|u-x|-|v-y|$.

{\bf I.} The lower bound of $|\gamma(x)-\gamma(y)|$ is given by
$$
\frac1{\lambda}(|u-v|-|u-x|-|v-y|)-c\le0\le|\gamma(x)-\gamma(y)|.
$$

{\bf II.} The upper bound of $|\gamma(x)-\gamma(y)|$ is given by
\begin{align*}
&\lambda(|u-v|-|u-x|-|v-y|)+c-(|a-\gamma(y)|-|a-\gamma(x)|)
\\
&\qquad\qquad=\lambda(|u-v|-|u-x|-|v-y|)+c-\lambda(|v-y|-|u-x|)
\\
&\qquad\qquad=\lambda^2c-2\lambda|v-y|+c\ge c\ge0.
\end{align*}
% We must clarify the sense of some formulas in the case $\delta=0$.
% From~\eqref{eq1} and~\eqref{eq2}, we have
% $\delta e^{Kd_{i+1}/2}\le\lambda^2d_i$. With a fixed value of $d_{i+1}$,
% the expression in the left-hand side of the inequality goes to infinity
% as $\delta\to0$. On the other hand, it is bounded by the finite
% right-hand side. This means that the sequence ${d_i}$ decreases
% very rapidly for small $\delta$. In particular, for $\delta=0$,
% $d_1$ vanishes immediately, i.e., the projection of $\gamma_0$ is a
% single point. This can be easily seen by considering
% a quasi-geodesic in a tree.

\subsection{Achieving the displacement $\lambda c$}
We now describe a self-quasi-isometry $f$ of a ball $B$ in a tree that moves
the center $O$ a distance $\lambda c/2$. We assume that the radius of $B$
is greater than $\lambda c$. We note that the images of two points inside
the ball $B_1$ of radius $\lambda c$ with a center $O$ can be just the same
point. Let the quasi-isometry $f$ fix the boundary of $B_1$, and let
$|O-f(O)|=\lambda c/2$. The segment $[O,f(O)]$ is sent to the only point
$f(O)$. For any point $a$ of $\partial B_1$, we let $a'$ denote a projection
of $a$ on $[O,f(O)]$ and assume that the interval $[a,a']$ is linearly
stretched and sent to the interval $[a,f(O)]$. Such a map $f$ assigns only
one image to any point. It is easy to verify that $f$ is a quasi-isometry
because the distances between points can be diminished up to $0$ and are not
increased more than $\lambda$ times.

\section{Anti-Morse lemma}\label{antiMorseSection}

We have already proved that any quasi-geodesic $\gamma$ in a hyperbolic space
is at distance not more than $\lambda^2(A_1c+A_2\delta)$ from a geodesic
segment $\sigma$ connecting its endpoints. This estimate cannot be improved.
But the curious thing is that this geodesic belongs to a
$\ln\lambda$-neighborhood of the quasi-geodesic! We can therefore say that
any quasi-geodesic is $\ln\lambda$-quasiconvex. This upper bound can be
improved in some particular spaces: for example, any quasi-geodesic is
$c$-quasiconvex in a tree.

The proof of Theorem~\ref{SecondTheorem} (see the introduction) that we give
below is based on using
\begin{itemize}
\item Lemma~\ref{expC} (exponential contraction) to prove that at the
distance $\ln\lambda$ from the geodesic $\sigma$ is at most
$\lambda^2\ln\lambda$ and
\item an analogue of Lemma~\ref{expC} to prove that the length of a circle
of radius $R$ is at least $e^R$ (up to some constants).
\end{itemize}

\begin{lemma}\label{lem11}
Let $X$ be a hyperbolic metric space, $\gamma$ be a
$(\lambda,c)$-quasi-geodesic, and $\sigma$ be a geodesic connecting the
endpoints of $\gamma$. Let $(y_u,y_v)$ be an arc of $\gamma$ such that no
point of this arc is at distance less than $C_1\ln\lambda+C_2$ from
$\sigma$ and $y_u$ and $y_v$ are the points of the arc nearest $\sigma$.
Then the length of the projection of the arc $(y_u,y_v)$ on $\sigma$
does not exceed $\max(8\delta,C_3\ln\lambda)$ (with some well-chosen
constants $C_1$, $C_2$, and $C_3$ depending linearly on $c$).
\end{lemma}

\begin{proof}
By the definition of a quasi-geodesic, we have
$$
\frac{|u-v|}{\lambda}-c\le|y_u-y_v|\le\lambda|u-v|+c.
$$
On the other hand,
$$
|y_u-y_v|\le|y_u-y_u'|+|y_u'-y_v'|+|y_v'-y_v|,
$$
where $y'_u$ and $y'_v$ are the projections of $y_u$ and $y_v$ on $\sigma$.
We adjust the constants $C_1$ and $C_2$ such that
$$
C_1\ln\lambda+C_2=\frac{19\delta^2}K
\ln\frac{8\delta\lambda^4}{\Delta}+\Delta+58\delta,
$$
where $\Delta=2c$ (such a choice allows applying Lemma~\ref{qgLen}). We
apply the lemma on exponential contraction (we assume that the length of the
arc is rather large for using the estimate with an exponential factor and
not to treat the obvious case where the length of the projection is
$8\delta$). We let $l(y_u, y_v)$ denote the $\Delta$-length of the arc
$(y_u,y_v)$:
$$
|y'_u-y'_v|\le l(y_u,y_v)e^{-K(r-\Delta-58\delta)/\delta}=
\frac1{2\lambda^4}l(y_u,y_v).
$$
Combining all these inequalities and using Lemma~\ref{qgLen}, we obtain
\begin{align*}
\frac{|u-v|}{\lambda}-c\le|y_u-y_v|&\le
\frac8K\ln\sqrt[4]{2}\lambda+\frac1{8\lambda^4}l(y_u, y_v)
\\
&\le\frac8K\ln\sqrt[4]{2}\lambda+4\lambda^2\frac1{8\lambda^4}|y_u-y_v|
\\
&\le\frac8K\ln\sqrt[4]{2}\lambda+\frac1{2\lambda^2}(\lambda|u-v|+c).
\end{align*}
We therefore conclude that $|y_u-y_v|\le C_3\lambda^2\ln\lambda$, hence
$l(y_u,y_v)\le C_3\lambda^4\ln\lambda$, and, finally, the length of the
projection of the arc $(y_u,y_v)$ of $\gamma$ does not exceed
$\max(8\delta,C_3\ln\lambda)$.
\end{proof}

\begin{proof}[Proof of Theorem~\ref{SecondTheorem}]
The proof follows directly from Lemma~\ref{lem11}. Because we have already
proved that for every point $z'\in\sigma$, there exists a point $z\in\gamma$
such that the projection of $z$ on $\sigma$ is at distance not more than
several times $c+\delta$ from $z'$. For simplicity, we therefore assume that
for any point of $\sigma$, there exists a point of $\gamma$ projecting on
this point.

If the distance between $z$ and $z'$ is less than $C_1\ln\lambda$ for some
constant $C_1=C_1(c,\delta)$ (the value of $C_1$ can be found from
Lemma~\ref{lem11}), then the statement is already proved. If not, then we
take an arc $(y_u,y_v)$ of $\gamma$ containing the point $z$ such that the
endpoints $y_u$ and $y_v$ are at the distance $C_1\ln\lambda$ from $\sigma$
and these points are the points of this arc that are nearest $\sigma$.
Hence, by the Lemma~\ref{lem11}, the length of the projection (which
includes $z$) of the arc $(y_u,y_v)$ does not exceed $C_4\ln\lambda$.
Therefore, the distance from $z$ to $y_u$ (and $y_v$) is not greater than
$(C_1+C_4)\ln\lambda$.
\end{proof}

\section{Geodesically rich spaces}

\begin{definition}
A metric space $X$ is said to be geodesically rich if there exist constants
$r_0$, $r_1$, $r_2$, $r_3$, and $r_4$ such that
\begin{itemize}
\item for every pair of points $p$ and $q$ with $|p-q|\ge r_0$, there
exists a geodesic $\gamma$ such that $d(p,\gamma)<r_1$ and
$|d(q,\gamma)-|q-p||<r_2$ and
\item for any geodesic $\gamma$ and any point $p\in X$, there exists a
geodesic $\gamma'$ passing in a $r_3$-neighborhood of the point $p$ and such
that $d(p,\gamma)$ differs from the distance between $\gamma'$ and $\gamma$
by not more than $r_4$.
\end{itemize}
\end{definition}

\begin{example}
A line and a ray are not geodesically rich. Both of them satisfy the second
condition in the definition, but not the first.
\end{example}

\begin{example}
Nonelementary hyperbolic groups are geodesically rich. We prove this later.
\end{example}

Any $\delta$-hyperbolic metric space $H$ can be embedded isometrically in a
geodesically-rich $\delta$-hyperbolic metric space $G$ (with the same
constant of hyperbolicity). We take a 3-regular tree with a root $(T,O)$,
assume that $G=H\times T$, and set the metric analogously to a real tree:
\begin{itemize}
\item the distance between points in the subspace $(H,O)$ equals the
distance between the corresponding points in $H$;
\item the distance between other points equals the sum of the three
distances from the points to their projections on $(H,O)$ and between their
projections on $(H,O)$.
\end{itemize}
It is easy to show that the space $G$ is $\delta$-hyperbolic and
geodesically rich. But such a procedure completely changes the ideal
boundary of the space. We therefore ask another question:

\begin{question}
Is it possible to embed a $\delta$-hyperbolic metric space $H$ isometrically
in a geodesically rich $\delta$-hyperbolic metric space $G$ with an
isomorphic boundary?
\end{question}

\begin{lemma}\label{hypGr1}
Let $G$ be a nonelementary hyperbolic group. Then there exist constants
$c_1$ and $c_2$ such that for any point $p\in G$ and any geodesic
$\gamma\in G$ such that $d(p,\gamma\ge c_1$, there exists a geodesic
$\gamma'$ with a point $q$ minimizing (up to a constant times $\delta$) the
distance to $\gamma$ and $|p-q|\le c_2$.
\end{lemma}

\begin{proof}
By symmetry, we can assume that $p$ is the unity of the group $G$. We supply
the ideal boundary $G(\infty)$ with a visual distance. Because $G$ is a
nonelementary group, its ideal boundary $G(\infty)$ has at least three
points (hence, infinitely many points).

We first prove by contradiction that there exists a $\varepsilon$ such that
for every pair of points $\xi$ and $\eta$ of $G(\infty)$, the union of the
two balls of radius $\varepsilon$ with the centers $\xi$ and $\eta$ does not
cover the whole ideal boundary. On the contrary, we suppose that there exist
two sequences of points $\xi_n$ and $\eta_n$ such that the union of
$B(\xi_n,1/n)$ and $B(\eta_n,1/n)$ includes $G(\infty)$. By compactness, we
can assume that $\xi_n\to\xi$ and $\eta_n\to\eta$, and we find that
$G(\infty)$ belongs to the union of $B(\xi,2/n)$ and $B(\eta,2/n)$. Hence,
the ideal boundary contains only the two points $\xi$ and $\eta$, which
contradicts the assumption that $G$ is nonelementary.

Let $c_1$ be a constant such that if a geodesic $\gamma$ is at a distance at
least $c_1$ from the point $p$, then the visual distance between its
endpoints (at infinity) is less than $\varepsilon/2$. We now take two points
$\xi$ and $\eta$ of $G(\infty)$ outside a $\varepsilon/4$-neighborhood of
$\gamma(\infty)$ such that $|\xi-\eta|>\varepsilon$ (the preceding argument
established that such a choice is possible). Let $\gamma'$ be a geodesic
with the endpoints $\xi$ and $\eta$. Hence, $d(p,\gamma')<c_1$. Applying
Lemma~\ref{hypGr2} completes the proof.
\end{proof}

\begin{lemma}\label{hypGr2}
Let $X$ be a $\delta$-hyperbolic space. Then for every $\varepsilon>0$, there
exist constants $c_1$ and $c_2$ such that for every pair of geodesics
$\gamma$ and $\gamma'$ and every point $p$ such that $d(p,\gamma)<c_1$
and visual distance between the endpoints $\gamma(\infty)$ and
$\gamma'(\infty)\ge\varepsilon$, there exists a point $q$ on $\gamma$
minimizing the distance to $\gamma'$ up to some constant times $\delta$ and
such that $|p-q|\le c_2$.
\end{lemma}

\begin{proof}
By Lemma~\ref{VDP}, we can replace the point $p$ with its projection $p'$
on the geodesic $\gamma$. Let $a'$ and $b'$ be the projections on $\gamma$
of the endpoint $a=\gamma'(-\infty)$ and the point $b$ of $\gamma'$ that
minimizes the distance from $\gamma'$ to $\gamma$.

We consider two sequences $x_n$ and $y_n$ of points respectively on $aa'$
and $a'\gamma(+\infty)$ such that $\lim_{n\to\infty}x_n=a$ and
$\lim_{n\to\infty} y_n=\gamma(+\infty)$. We let $a'_n$ denote the
projections of $x_n$. Obviously, $a'_n\to a'$ as $n\to\infty$. By the
definition of Gromov's product, $(x|y)_{p'}=\lim_{n\to\infty}
(x_n|y_n)_{p'}$. Using Lemma~\ref{orthTriangle}, we now estimate
$(x_n|y_n)_{p'}$:
\begin{align*}
(x_n|y_n)_{p'}&=\frac12(|p'-x_n|+|p'-y_n|-|x_n-y_n|)
\\
&\le\frac12(|p'-a'_n|+|a'_n-x_n|+8\delta+
|p'-y_n|-|a'_n-x_n|-|a'_n-y_n|+2\delta).
\end{align*}

Now, if $p'$ is between $a'$ and $b'$, then $(x_n|y_n)_{p'}\le5\delta$;
otherwise (we assume that $p'$ is closer to $a'$, i.e., the order of points
on $\gamma$ is $p',a',b'$), $(x_n|y_n)_{p'}\le|p'-a'|+5\delta$.

Therefore, to finish the proof, we must now prove that the point $a'$ is not
far from $ab$. We apply Lemma~\ref{orthTriangle} once more to the triangle
$aa'b'$ and obtain $d(a',ab')\le2\delta$. Hence, because the triangle $abb'$
is $\delta$-thin, the distance from $a'$ to $ab$ or $bb'$ is not greater
than $3\delta$. In the first case, the statement is proved immediately. In
the second case, we note that $bb'$ is a perpendicular to $ab'$ and hence
$d(a'b')\le2d(a',bb')\le6\delta$. Therefore, $a'$ in this case is near the
projection of the point of $ab$ that is nearest $ab'$, which completes the
proof.
\end{proof}

\begin{lemma}\label{lem14}
Let $G$ be a nonelementary hyperbolic group. Then there exist constants
$c_0$, $c_1$, and $c_2$ such that for every two points $p$ and $q$ in the
group $G$ with $|p-q|>r_0$, there exists a geodesic $\gamma$ such that
$d(p,\gamma)\le r_1$ and $||p-q|-d(q,\gamma)|\le r_2$.
\end{lemma}

\begin{figure}[!ht]
\center{\includegraphics[width=0.5\linewidth]{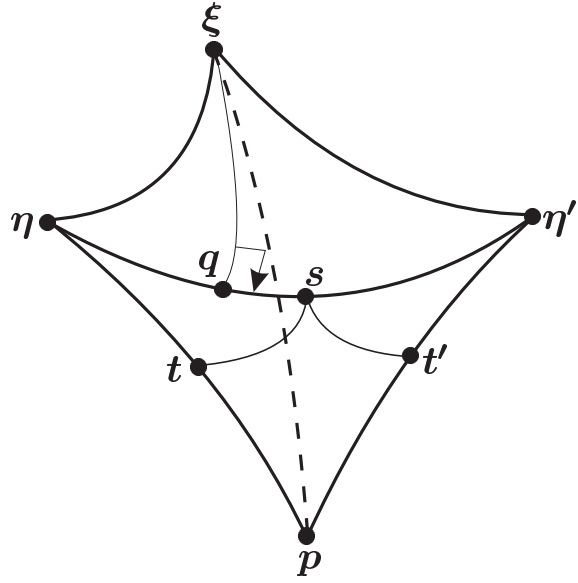}}
%\center{\includegraphics[width=0.5\linewidth]{lem14.eps}}
\caption{Illustration for Lemma~\ref{lem14}.}
\label{lem14fig}
\end{figure}

\begin{proof}
We first assume that $p$ is the unity of the group. We argue by
contradiction: we suppose that the statement is false, i.e., there exists a
sequence of points ${q_n}$ such that $|q_n-p|\to\infty$ as $n\to\infty$, and
all pairs $p$ and $q_n$ do not satisfy the conditions in the lemma. We
suppose that $\xi$ is a limit point of this sequence. As in the proof of
Lemma~\ref{hypGr1}, we supply the boundary of the group with a visual
metric. And the same arguments provide that there exist $\varepsilon>0$ and
points $\eta$ and $\eta'$ on the ideal boundary $G(\infty)$ such that that
the pairwise visual distances between $\xi$, $\eta$, and $\eta'$ are greater
than $\varepsilon$ (see Fig.~\ref{lem14fig}). We show that the geodesic
$\gamma$ with the endpoints $\eta$ and $\eta'$ satisfies the conditions in
the lemma, which leads to the contradiction.

In what follows, we write $\xi$, $\eta$, and $\eta'$ but assume that we
consider three sequences of points converging to the corresponding points of
the ideal boundary. The triangle $p\eta\eta'$ is $\delta$-thin. We take a
point $s$ of $\eta\eta'$ such that $d(s,p\eta)\le\delta$ and
$d(s,p\eta')\le\delta$. We let $t$ and $t'$ denote projections of $s$
respectively on $p\eta$ and $p\eta'$. By the triangle inequality, we have
$$
|\eta-t|+|\eta'-t'|-2\delta\le|\eta-\eta'|\le|\eta-t|+|\eta'-t'|+2\delta.
$$
By hypothesis,
$$
\vd_p(\eta,\eta')=e^{-(\eta|\eta')_p}>\varepsilon.
$$
Hence,
$$
|p-\eta|+|p-\eta'|-|\eta-\eta'|<2\varepsilon_0,
$$
where $\varepsilon_0=-\ln\varepsilon$

Combining the two inequalities, we obtain $|p-t|+|p-t'|\le2(\varepsilon_0+
\delta)$ and $d(p,\eta\eta')\le2\varepsilon_0+3\delta$. The same arguments
applied to the triangles $p\eta\xi$ and $p\eta'\xi$ show that the distance
from the point $p$ to the geodesics $\eta\xi$ and $\eta'\xi$ also does not
exceed $2\varepsilon_0+3\delta$. We let $p_1$, $p_2$, and $p_3$ denote the
respective projections of $p$ on $\eta\eta'$, $\eta\xi$, and $\eta'\xi$
and $q$ denote the projection of $\xi$ on $\eta\eta'$. By the triangle
inequality, $|p_1-p_2|\le|p_1-p|+|p-p_2|\le2(2\varepsilon_0+3\delta)$.
Applying Lemma~\ref{orthTriangle} to the triangles $q\xi\eta$ and
$q\xi\eta'$, we find that the point $q$ is not farther than $2\delta$ from
both $\eta\xi$ and $\eta'\xi$. Therefore, both $p_1$ and $q$ are at bounded
distances from $\eta\xi$ and $\eta'\xi$, and we can apply
Lemma~\ref{thinTriangle}, whence it follows that $p_1$ and $q$ are near
each other at a distance of the order $\varepsilon_0+\delta$.
\end{proof}

%\begin{lemma}
% Let $X$ be a $\delta$-hyperbolic space containing two geodesics $\gamma$
% and $\gamma'$ at a distance greater than $2\delta$ from each other. Then
% the length of the projection of $\gamma'$ on $\gamma$ does not exceed
% $6\delta$.
%\end{lemma}

%\begin{proof}
% Let $a$ and $b$ denote the endpoints of $\gamma'$ and $p$ and $q$ denote
% their projections on $\gamma$. The triangle $apq$ is $\delta$-thin. We
% find a point $r$ on $aq$ such that the distances from $r$ to $ap$ and $pq$
% do not exceed $\delta$. Because the distance between $\gamma$ and
% $\gamma'$ is greater than $2\delta$, we have $d(r,ab)>\delta$. But the
% triangle $abq$ is $\delta$-thin, and $d(r, bq)$ hence does not exceed
% $\delta$. Let $r_1$ and $r_2$ denote the projections of $r$ on $ap$ and
% $bq$. We have
% \begin{itemize}
% \item $|r_1-r_2|\le2\delta$,
% \item $d(r_1,\gamma)\le2\delta$, and $d(r_1,\gamma)\le2\delta$.
% \end{itemize}
% Hence, the length of the projection of $\gamma'$ on $\gamma$ is not
% greater than $6\delta$.
%\end{proof}

\begin{lemma}\label{VDP}
Let $X$ be a $\delta$-hyperbolic space, $\xi$ and $\eta$ be two points of
the ideal boundary $\partial X$, and $p$ and $p'$ be two points such that
$d(p,p')=D$. Then the visual distances between $\xi$ and $\eta$ from the
points $p$ and $p'$ satisfy the inequality
$$
\vd_{p'}(\xi,\eta)\le e^D\vd_p(\xi,\eta).
$$
\end{lemma}

\begin{proof}
By definition, Gromov's product of $x$ and $y$ in $p$ is
$$
(x|y)_p=\frac12(|p-x|+|p-y|-|x-y|).
$$
We have the same equality for $x$, $y$, and $p'$. Hence,
$$
|(x|y)_{p'}-(x|y)_p|=|\frac12(|p'-x|+|p'-y|-|p-x|-|p-y|)|\le|p-p'|.
$$
The last inequality follows from the triangle inequality. Therefore, by the
definition of a visual metric,
$$
\vd_{p'}(\xi,\eta)=e^{(\xi|\eta)_{p'}}\le
e^{(\xi|\eta)_p+|p-p'|}=e^D\vd_p(\xi,\eta).
$$
\end{proof}

\section{Quasi-isometries fixing the ideal boundary}\label{idBoundary}

We now give some estimates of the displacement of points in geodesically
rich spaces under quasi-isometries that fix the ideal boundary. We do not
yet know whether these results are optimal.

\begin{remark}
Let $X$ be a metric space satisfying the first condition in the definition
of geodesically rich. Let $f:X\to X$ be a $(\lambda,c)$-self-quasi-isometry
fixing the boundary $\partial X$. Then for sufficiently large $\lambda$ and
any point $O\in X$, $d(f(O),O)\le H(\lambda,c,\delta)+r_2$, where the
constant $C_1$ depends only on the space $X$.
\end{remark}

\begin{proof}
For any point $O$, $r_1\le H(\lambda,c,\delta)$ for sufficiently large
$\lambda$ if $d(O,f(O))<r_0$. Otherwise, let $\gamma$ be a geodesic such
that $d(O,\gamma)\le r_1$ and $d(f(O),\gamma)>d(O,f(O))-r_2$. Because
$f(\gamma)$ is a quasi-geodesic with the same endpoints as $\gamma$, the
quasi-geodesic lies near $\gamma$: $f(\gamma)\subset U_H(\gamma)$. Combining
all the arguments, we obtain
$$
d(O,f(O))\le d(f(O),\gamma)+r_2\le H+r_2.
$$
\end{proof}

We do not know if it is possible to improve this upper bound in the general
case. But in the case of a geodesically rich space, we can improve the
bound from $\lambda^2$ to $\lambda\ln\lambda$.

\begin{thnn}[see Theorem~\ref{GRSTheorem} in the introduction]
Let $X$ be a $(r_1,r_2)$-geodesically rich $\delta$-hyperbolic metric space
and $f$ be a $(\lambda,c)$-self-quasi-isometry fixing a boundary
$\partial X$. Then for any point $O\in X$,
$d(O,f(O))\le\max(r_0,\lambda (r_3+c+c_1\ln\lambda)+r_1+r_2+r_4)$.
\end{thnn}

\begin{proof}
Because $f$ fixes the boundary of $X$ and by the anti-Morse lemma, a
$(c_1\ln\lambda)$-neighborhood (where $c_1=c+\delta$) of an image
$f(\sigma)$ of any geodesic $\sigma$ includes $\sigma$: $\sigma\subset
V_{c_1\ln\lambda}(f(\sigma))$. All the constants $r_0$, $r_1$, $r_2$,
$r_3$, and $r_4$ are the same constants as in the definition of a
geodesically rich space. We take an arbitrary point $O\in X$. We assume
that $d(O,f(O))\ge r_0$ because otherwise there is nothing to prove. There
exists a geodesic $\gamma$ such that $d(\gamma,O)\le r_1$ and
$|d(O,f(O))-d(f(O),\gamma)|\le r_2$, and there also exists a geodesic
$\gamma'$ such that $f(O)$ lies in $r_3$-neighborhood of $\gamma'$ and
such that $f(O)$ is (up to $r_4$) the point of $\gamma'$ that is nearest
$\gamma$.

% We take a geodesic $\gamma$ in $X$. Let $O$ be a point of $\gamma$ such
% that $f(O)$ is the point of $f(\gamma)$ most distant from $\gamma$.
% Let $\gamma'$ be a geodesic with different endpoints (at infinity)
% such that $f(O)$ lies in an $r_1$-neighborhood of $\gamma'$ and such
% that $f(O)$ is the point of $\gamma'$ nearest $\gamma$.

Because $\gamma'\subset V_{c_1\ln\lambda}(f(\gamma'))$, there exists a
point $O'$ of $\gamma'$ such that $|f(O')-f(O)|\le r_3+c_1\ln\lambda$.
Now, $d(f(O),\gamma)\le d(O',\gamma)+r_4\le|O'-O|+r_1+r_4$, and by the
definition of a quasi-isometry, $|O'-O|\le\lambda(|f(O')-f(O)|+c)\le
\lambda(r_3+c+c_1\ln\lambda)$. Hence, $d(f(O),\gamma)\le
\lambda(r_3+c+c_1\ln\lambda)+r_1+r_4$. Finally, we conclude that
$d(O,f(O))\le d(f(O),\gamma)+r_2\le\lambda(r_3+c+c_1\ln\lambda)+
r_1+r_2+r_4$.
\end{proof}

\section{Acknowledgment}
I am thankful to Professor P.~Pansu for advising me through all the steps of
this research. This work was supported by the President of the Russian
Federation under the program of state support for leading scientific schools
(Grant No.~NSh-8462.2010.1).

\end{document}